\documentclass[preprint,12pt,3p]{elsarticle}
\usepackage{amsmath,amsthm,amsfonts,amssymb}
\usepackage{hyperref}
\usepackage{xcolor}
\usepackage{sectsty}
\usepackage{enumitem}
\definecolor{darkslategray}{rgb}{0.18, 0.31, 0.31}
\definecolor{warmblack}{rgb}{0.0, 0.26, 0.26}
\definecolor{astral}{RGB}{46,116,181}
\subsectionfont{\color{astral}}
\sectionfont{\color{astral}}

\newtheorem{theorem}{Theorem}[section]

\newtheorem{lemma}[theorem]{Lemma}
\newtheorem{definition}[theorem]{Definition}

\newtheorem{example}[theorem]{Example}

\newcommand{\0}{\bf{0}}

\newcommand{\C}{{\mathbb C}}
\newcommand{\mc}[1]{\mathcal {#1}}
\newcommand{\dg}{{\dagger}}

\newcommand{\n}{{*_N}}
\newcommand{\m}{{*_M}}

\newcommand{\lp}{{*_L}}
\newcommand{\1}{{*_1}}
\newcommand{\2}{{*_2}}

\journal{RACSAM}

\begin{document}

\begin{frontmatter}

\title{ \textcolor{warmblack}{\bf On reverse-order law of tensors and its application to additive results on Moore-Penrose inverse}}
\author{Krushnachandra Panigrahy$^\dag$$^a$, Debasisha Mishra$^\dag$$^b$}

\address{               $^{\dag}$Department of Mathematics,\\
                        National Institute of Technology Raipur,\\
                        Raipur, Chhattisgarh, India.\\
                        \textit{E-mail$^a$}: \texttt{kcp.224\symbol{'100}gmail.com }\\
                        \textit{E-mail$^b$}: \texttt{dmishra\symbol{'100}nitrr.ac.in. }
        }

\begin{abstract}
The equality $(\mc{A}\n\mc{B})^\dg = \mc{B}^\dg \n \mc{A}^\dg $ for any two complex tensors $\mc{A}$ and $\mc{B}$ of arbitrary order, is called as the {\it reverse-order law} for the Moore--Penrose inverse of arbitrary order tensors via the Einstein product. Panigrahy {\it et al.} [Linear Multilinear Algebra; 68 (2020), 246-264.] obtained  several necessary and sufficient conditions  to hold the reverse-order law for the Moore--Penrose inverse of even-order tensors via the Einstein product, very recently. This notion is revisited here  among other results. In this context, we  present several new characterizations of the reverse-order law of arbitrary order tensors  via the same product. More importantly, we illustrate a result on the Moore--Penrose inverse of a sum of two tensors as an application of the reverse-order law which leaves an open problem. We recall the definition of the Frobenius norm and the spectral norm to illustrate a result for finding the additive perturbation bounds of the Moore--Penrose inverse under the Frobenius norm.  We conclude our paper with the introduction of the notion of sub-proper splitting for tensors which may help to find an iterative solution of a tensor multilinear system.
\end{abstract}


\end{frontmatter}

\section{Introduction}\label{sec1}
Many phenomena are modeled as multilinear systems in engineering and science such as isotropic and anisotropic elasticity are modeled \cite{lai2009} as multilinear systems in continuum physics and engineering. Tensor methods have been used to solve problems in quantum chemistry:  the fundamental Hatree-Fock equation is solved by Khoromskij, Khoromskaia, and Flad \cite{khoromskij2011} and the works on multidimensional operators in quantum models were done by Beylkin and Mohlenkamp \cite{beylkin2005}, \cite{beylkin2002}. Multidimensional
boundary and eigenvalue problems are solved by Hackbusch and
Khoromskij \cite{hackbush2007}, and Hackbusch, Khoromskij and Tyrtyshnikov \cite{hackbush2005} through separated representation and hierarchical Kronecker tensor from the underlying high spatial dimensions using a reduced low-dimensional tensor-product space.
A {\it tensor} is a multidimensional array. An  {\it $N$th-order tensor} is an element of ${\C}^{I_{1}\times\cdots\times I_{N}}$  which is the set of order $N$ complex tensors. Here  $I_{1}, ~I_{2},~ \cdots,~ I_{N}$ are dimensions of the first, second, $\cdots$, $N$th way, respectively. The {\it order} of a tensor is the number of dimensions. A zero-order tensor is a scalar, a first-order tensor is a vector while  a second-order tensor is a matrix.  Higher-order tensors are tensors of order three or higher. These are denoted by calligraphic letters like  $\mc{A}$. $a_{i j k}$ denotes an $(i,j,k)$th element of a third order tensor $\mc{A}$, and in general, $a_{i_{1}\cdots i_{N}}$ represents an $(i_{1},~\cdots,~i_{N})$th element of an $N$th order tensor $\mc{A}$. For more details, we refer to the recent books \cite{che2020}, \cite{qi2017}, \cite{smilde2004}, \cite{wei2016}  on tensors.\\

There has been active research on tensors for the past four decades.  For applications and tensor-based methods, we refer the readers to the survey papers \cite{khoromskij2012}, \cite{kolda2009}, \cite{lathauwer2009} and the references cited therein. But, research contributions on the theory and applications of generalized inverses of tensors are very little. The first work in this direction was reported in 2016 (see the article \cite{sun2016} by  Sun { \it et al.}) where the authors formally introduced a generalized inverse called the {\it Moore--Penrose inverse of an even-order tensor} via the Einstein product. The authors \cite{sun2016} then used the Moore--Penrose inverse to find the minimum-norm least-squares solution of some multilinear systems.  Behera and Mishra \cite{behera2017} continued the same study and proposed different types of generalized inverses in the next year. They initiated the study of the reverse-order law for different generalized inverses of tensors. In 2017, Jin {\it et al.} \cite{jin2017} again introduced the Moore--Penrose inverse of a tensor using another tensor product called {\it t-product} and applied it to derive the higher-order Gauss-Markov theorem. In the same year, Ji and Wei \cite{ji2017}  extended the Moore--Penrose inverse of an even-order tensor to the {\it weighted Moore--Penrose inverse} of an even-order tensor and established the relation between the minimum-norm least-squares solution of a multilinear system and the weighted Moore--Penrose inverse. In 2018, the same authors introduced another generalized inverse of an even-order square tensor called the {\it Drazin inverse} \cite{ji2018}. They used this inverse to find the Drazin inverse solution of the singular linear tensor equation. In the same year, Sun {\it et al.} \cite{sun2018}  introduced the {\it group inverse} of an even-order square tensor using a general product of tensors.  Very recently, Huang {\it et al.} \cite{huang2019} proposed a {\it tensor-based extreme learning machine}.  They used the Moore--Penrose inverse to obtain the tensor regression result. In 2018, Panigrahy and Mishra \cite{panigrahy2018} and  Stanimirovic {\it et al.} \cite{stanimirovic2018}  simultaneously improved the definition of the Moore--Penrose inverse of an even-order tensor to a tensor of any order via the same product. In 2018, Liang and Zheng \cite{liang2019} also proposed the same definition which appeared in the year 2019.  
The definition of the Moore--Penrose inverse of an arbitrary order tensor is recalled below.
\begin{definition}(Definition 1.1, \cite{panigrahy2018})\label{defmpi}
Let $\mc{X} \in {\C}^{I_{1}\times\cdots\times I_{N} \times J_{1}\times \cdots \times J_{M}}$. The tensor $\mc{Y} \in \\{\C}^{J_{1}\times\cdots\times J_{M} \times I_{1} \times\cdots\times I_{N}}$ satisfying the following four tensor equations:
\begin{eqnarray}
\mc{X}\m\mc{Y}\n\mc{X} &=& \mc{X};\label{mpeq1}\\
\mc{Y}\n\mc{X}\m\mc{Y} &=& \mc{Y};\label{mpeq2}\\
(\mc{X}\m\mc{Y})^{H} &=& \mc{X}\m\mc{Y};\label{mpeq3}\\
(\mc{Y}\n\mc{X})^{H} &=& \mc{Y}\n\mc{X},\label{mpeq4}
\end{eqnarray}
is defined as the \textbf{Moore--Penrose inverse} of $\mc{X}$, and is denoted by $\mc{X}^{\dg}$.
\end{definition}
\noindent In the above definition, $\m$ denotes the {\it Einstein product} \cite{einstein2007} of tensors, and is defined by 
\begin{equation*}\label{Eins}
(\mc{A}\m\mc{B})_{i_{1}\cdots i_{N}k_{1}\cdots k_{L}}
=\displaystyle\sum_{j_{1}\cdots j_{M}}a_{{i_{1}\cdots i_{N}}{j_{1}\cdots j_{M}}}b_{{j_{1}\cdots j_{M}}{k_{1}\cdots k_{L}}}
\end{equation*}
for tensors $\mc{A} \in \mathbb{C}^{I_{1}\times\cdots\times I_{N} \times J_{1}\times\cdots\times J_{M} }$ and $\mc{B} \in \mathbb{C}^{J_{1}\times\cdots\times J_{M} \times K_{1} \times\cdots\times K_{L} }$.
In the case of an even-order invertible tensor, the Definition \ref{defmpi} coincides with the notion of the inverse which was first introduced by Brazell {\it et al.} \cite{brazell2013}. They also showed that such an inverse can be computed using the singular value decomposition of the same tensor (see the celebrated result in Lemma 3.1, \cite{brazell2013}). A generalized inverse of a tensor, as its name indicates, is a generalization of the notion of the tensor inverse. It is a  tensor that exists for a larger class of tensors than the ordinary inverse does; has some properties of the tensor inverse, and for a given even-order square invertible tensor it reduces to the ordinary inverse (see  \cite{wang2018} and \cite{wei2018} for theory of generalized inverses of matrices). The idea of introducing generalized inverses of tensors originates from the necessity of finding a solution of a given system of linear equations (see \cite{behera2017}, \cite{ji2018},  \cite{ji2017}, \cite{sun2016}). \\

This paper is in connection with an open problem stated in the last section of \cite{behera2017} for even-order tensors, and is a continuation of the very recent works done in \cite{panigrahy2020} and \cite{panigrahy2018}.   The same open problem is restated next in the setting of tensors of any order.\\
 
 {\bf Problem 1.} When does  $(\mc{A}\n\mc{B})^\dg = \mc{B}^\dg \n
 \mc{A}^\dg $ for  any two tensors $\mc{A}\in{\C}^{I_{1}\times\cdots \times I_{M}\times J_{1}\times\cdots \times J_{N}}$ and $\mc{B}\in{\C}^{J_{1}\times\cdots \times J_{N}\times K_{1}\times\cdots \times K_{L}}$?
  \\
 
 Note that Panigrahy {\it et al.} \cite{panigrahy2020} attempted first the above Problem 1, very recently, and provided various necessary and sufficient conditions for the same problem, but for even-order tensors only.  Nevertheless, we aim to
  \begin{enumerate}
 \item 
 establish 
 some new necessary and sufficient conditions for the above-stated problem which is well-known as the {\it reverse-order law} for the Moore--Penrose inverse of arbitrary order tensors via the Einstein product,
 
\item compute the Moore--Penrose inverse of the sum of two arbitrary order tensors,

\item find
 the additive perturbation bounds of the Moore--Penrose inverse under the Frobenius norm, and to
 
 \item introduce the concept of sub-proper splittings for tensors without using the notion of the range and null space of a tensor which may help to find an iterative solution of a tensor multilinear system.
  \end{enumerate}
 
In this context, the paper is organized as follows. Section 2  collects various useful definitions and results. The next section which contains all our main results, partially fulfills the above-stated objectives and provides some open problems for future research interest. It has two subsections. The first subsection  deals with the reverse-order law tensors of any order. It also discusses a few properties of different generalized inverses of a tensor, and then illustrates a new expression of the Moore--Penrose inverse of the Einstein product of two tensors in terms of the Einstein product of the Moore--Penrose inverse of a tensor and a generalized inverse of a tensor. The next subsection presents a result which is useful to obtain the additive Moore--Penrose inverse of tensors as an application of the reverse-order law and leaves an open problem. It then attempts the problem of fining the additive perturbation bounds of the Moore--Penrose inverse.  Finally, it provides an answer to the last objective of this paper as mentioned in the previous paragraph before moving to the concluding section.

 \section{Prerequisites}\label{sec2}
Here, we collect all those remaining definitions and earlier results which will be used to prove the main results in the next section.
 We begin with the definition of an identity tensor. A tensor $\mc{I}\in {\C}^{I_{1}\times\cdots\times I_{N}\times I_{1}\times\cdots\times I_{N}}$ with entries  $ (\mc{I})_{i_{1} \cdots i_{N}j_{1}\cdots j_{N}} = \prod_{k=1}^{N} \delta_{i_{k} j_{k}}$ is  called   an {\it identity tensor}  if
$ \delta_{i_{k}j_{k}}= \begin{cases}
 1, \text{ if } i_{k} = j_{k}\\
 0, \text{ otherwise}
 \end{cases}$.
 The {\it conjugate transpose} of a tensor $\mc{A}\in {\C}^{I_{1}\times\cdots\times I_{M}\times J_{1}\times \cdots \times J_{N}}$ is denoted by $\mc{A}^{H}$, and is defined as $(\mc{A}^{H})_{j_{1}\hdots j_{N}i_{1}\hdots i_{M}}=\overline{a}_{i_{1}\hdots i_{M}j_{1}\hdots j_{N}},$ 
where the over-line stands for the  conjugate of $a_{i_{1}\hdots i_{M}j_{1}\hdots j_{N}}$. If the tensor $\mc{A}$ is real, then its {\it transpose} is denoted by $\mc{A}^{T}$, and is defined as $(\mc{A}^{T})_{j_{1}\hdots j_{N}i_{1}\hdots i_{M}}=a_{i_{1}\hdots i_{M}j_{1}\hdots j_{N}}$.  We next present the definition of a unitary and an idempotent tensor. A tensor $\mc{A}\in \mathbb{C}^{I_{1}\times\cdots\times I_{N} \times I_{1}\times\cdots\times I_{N}}$ is {\it unitary} if $\mc{A}\n\mc{A}^{H}=\mc{A}^{H}\n \mc{A}=\mc{I}$, and {\it idempotent}  if $\mc{A}\n\mc{A}= \mc{A}.$ 
If  $\mc{A}=\mc{A}^{H}$ for a tensor $\mc{A}\in\mathbb{C}^{I_{1}\times\cdots\times I_{N} \times I_1 \times\cdots\times
I_N}$, then $\mc{A}$ is {\it Hermitian}. If $\mc{A}= - \mc{A}^{H}$, then it is {\it skew-Hermitian}. Liang and Zheng \cite{liang2019} defined a bijective map $\phi$ from the tensor space ${\C}^{I_{1}\times\cdots\times I_{M}\times J_{1}\times\cdots\times J_{N}}$ into the matrix space ${\C}^{(I_{1}\cdot I_{2}\cdots I_{M})\times (J_{1}\cdot J_{2}\cdots J_{N})}$ by 
\begin{equation}\label{tnsr2mtrx}
    \mc{A}=(\mc{A}_{i_{1}\hdots i_{M} j_{1}\hdots j_{N}})\xrightarrow{\phi}A=(A_{ivec(\mathbf{i},\mathbb{I}),ivec(\mathbf{j},\mathbb{J})}),
\end{equation}
where
\begin{eqnarray*}
    ivec(\mathbf{i},\mathbb{I})&=&i_{1}+\sum_{r=2}^{M}(i_{r}-1)\prod_{u=1}^{r-1}I_{u},~\mathbf{i}:=\{i_{1},~\cdots,~i_{M}\},~\mathbb{I}:=\{I_{1},\cdots,I_{M}\},\\
    ivec(\mathbf{j},\mathbb{J})&=&j_{1}+\sum_{s=2}^{N}(j_{s}-1)\prod_{v=1}^{s-1}J_{v},~\mathbf{j}:=\{j_{1},~\cdots,~j_{N}\},~ \mathbb{J}:=\{J_{1},\cdots,J_{N}\}.
\end{eqnarray*}
Here, $\mathbb{I}:=\{I_{1},\cdots,I_{M}\}$ and $\mathbb{J}:=\{J_{1},\cdots,J_{N}\}$ are respectively referred to as the {\it row mode} and the {\it column mode} of $\mc{A}$. Next, we recall the definition of unfolding rank of a tensor.

\begin{definition}[Definition 3.1, \cite{liang2019}]
For a tensor $\mc{A}\in {\C}^{I_{1}\times \cdots\times I_{M}\times J_{1}\times \cdots\times J_{N}}$, and the map $\phi$ as defined in \eqref{tnsr2mtrx}, the unfolding rank of the tensor $\mc{A}$ is denoted by $rank_{U}(\mc{A})$ and defined as $rank_{U}(\mc{A})=rank(\phi(\mc{A}))$.
Particularly, if $rank_{U}(\mc{A}) = m$
$(rank_{U}(\mc{A}) = n)$, we say that $\mc{A}$ is full row (column) rank, where $m=I_{1}\cdot I_{2}\cdots I_{m}$ and $n=J_{1}\cdot J_{2}\cdots J_{N}$.
\end{definition}

The following results are useful to prove our main results. The first one is about the singular value decomposition (SVD) of a tensor proved
in Theorem 3.17, \cite{brazell2013} for a real tensor. The authors of \cite{sun2016} stated the same result for complex even order tensors. We next present the SVD for arbitrary order tensors.

\begin{lemma}{(Theorem 3.2., \cite{liang2019})}
 A tensor $\mc{A} \in
\mathbb{C}^{I_{1}\times\cdots\times I_{M} \times J_{1} \times\cdots\times
J_{N}}$ with $rank_{U}(\mc{A})=r$
 can be decomposed  as $$\mc{A} = \mc{U}\m\mc{B}\n\mc{V}^{H},$$
 where $\mc{U} \in \mathbb{C}^{I_{1}\times\cdots\times I_{M} \times I_{1} \times\cdots\times I_{M}}$ and
 $\mc{V} \in  \mathbb{C}^{J_{1}\times\cdots\times J_{N} \times J_{1} \times\cdots\times J_{N}}$ are unitary
 tensors, and
 $\mc{B} \in \mathbb{R}^{I_{1}\times\cdots\times I_{M} \times J_{1} \times\cdots\times J_{N}}$ is a diagonal
 tensor defined by
 \begin{equation*}
     \mc{B}_{i_{1}\cdots i_{M}j_{1}\cdots j_{N}}=\begin{cases}
     \sigma_{ivec(\mathbf{i},\mathbb{I})}>0,&\text{if } ivec{(\mathbf{i},\mathbb{I})}=ivec{(\mathbf{j},\mathbb{J})}\in\{1,~2,~\cdots,~r\},\\
     0,&\text{otherwise},
     \end{cases}
 \end{equation*}
 where $\sigma_{1}\geq\sigma_{2}\geq\cdots\geq\sigma_{r}>0$ are singular values of the matrix $\phi(\mc{A})$.
 \end{lemma}
The proof of Theorem 3.4,
\cite{liang2019}  contains the fact that $\mc{A}^{\dg} =
\mc{V}\n\mc{B}^{\dg}\m\mc{U}^{H}$.

\begin{lemma}(Lemma 2.2, \cite{panigrahy2018})\label{ir}
Let  $\mc{P}\in {\C}^{I_{1}\times\cdots\times I_{N}\times I_{1}\times \cdots \times I_{N}}$ be Hermitian. Then 
\begin{itemize}
   \item[(i)] If $\mc{P}\n\mc{Q}=\mc{Q}$ for $\mc{Q}\in {\C}^{I_{1}\times\cdots\times I_{N}\times J_{1}\times \cdots \times J_{M}}$, then $\mc{Q}^{\dg}\n\mc{P}=\mc{Q}^{\dg}$.\label{ir1} 
   \item[(ii)] If $\mc{Q}\n\mc{P}=\mc{Q}$ for $\mc{Q}\in {\C}^{J_{1}\times \cdots \times J_{M}\times I_{1}\times\cdots\times I_{N}}$, then $\mc{P}\n\mc{Q}^{\dg}=\mc{Q}^{\dg}$.\label{ir2}
\end{itemize}
\end{lemma}

Next, we restate a few results for tensors (arbitrary order) which are proved by Panigrahy {\it et al.} \cite{panigrahy2020} very recently for even-order tensors. 
The very first result stated below extends Theorem 3.25, \cite{panigrahy2020} to any two tensors, and provides necessary and sufficient conditions for the reverse-order law which is mentioned in Problem 1 in the introduction section.
\begin{theorem}\label{rev1}
Let $\mc{A}\in \mathbb{C}^{I_{1}\times\cdots\times I_{M} \times J_{1}\times\cdots\times J_{N} }$ and $\mc{B}\in \mathbb{C}^{J_{1}\times\cdots\times J_{N} \times K_{1}\times\cdots\times K_{L}}$. Then $(\mc{A}\n\mc{B})^{\dg} = \mc{B}^{\dg} \n \mc{A}^{\dg}$ if and only if 
\begin{equation}\label{reveq1}
\mc{A}^{\dg}\m\mc{A}\n\mc{B}\lp\mc{B}^{H} \n\mc{A}^{H} = \mc{B}\lp\mc{B}^{H} \n \mc{A}^{H}
\end{equation}
and
\begin{equation}\label{reveq2}
\mc{B}\lp\mc{B}^{\dg}\n\mc{A}^{H}\m\mc{A}\n\mc{B} = \mc{A}^{H}\m\mc{A} \n \mc{B}.
\end{equation}
 \end{theorem}
The next two lemmas are again the modified version of Lemma \ref{lm3.13} and Lemma \ref{lm3.14},  \cite{panigrahy2020}, and are frequently used to derive some new reverse-order laws. The first one provides a  necessary and sufficient condition for the commutative property of $\mc{A}^{\dg}\m\mc{A}$ and $\mc{B}\lp\mc{B}^{H}$.

\begin{lemma}\label{lm3.13}
Let $\mc{A} \in \mathbb{C}^{I_{1}\times\cdots\times I_{M} \times J_{1}\times\cdots\times J_{N} }$ and $\mc{B}\in \mathbb{C}^{J_{1}\times\cdots\times J_{N} \times K_{1}\times\cdots\times K_{L} }$. Then, $$\mc{A}^{\dg} \m \mc{A} \n \mc{B} \lp\mc{B}^{H} \n\mc{A}^{H}=\mc{B}\lp\mc{B}^{H} \n\mc{A}^{H}$$ if and only if $\mc{A}^{\dg}\m\mc{A}$ commutes with $\mc{B}\lp\mc{B}^{H}$.
\end{lemma}
Similarly, the next one  presents a  sufficient condition for the commutative property of $\mc{A}^{H}\m\mc{A}$ and $\mc{B}\lp\mc{B}^{\dg}$.

\begin{lemma}\label{lm3.14}
Let $\mc{A} \in \mathbb{C}^{I_{1}\times\cdots\times I_{M} \times J_{1}\times\cdots\times J_{N} }$ and $\mc{B}\in \mathbb{C}^{J_{1}\times\cdots\times J_{N} \times K_{1}\times\cdots\times K_{L} }$. Then,  $$\mc{B}\lp\mc{B}^{\dg}\n\mc{A}^{H}\m\mc{A}\n\mc{B}=\mc{A}^{H}\m\mc{A} \n \mc{B}$$ if and only if $\mc{A}^{H}\m\mc{A}$ commutes with $\mc{B}\lp\mc{B}^{\dg}$.
\end{lemma}
  Theorem \ref{rev1} together with Lemma \ref{lm3.13} and Lemma \ref{lm3.14} yields the following outcome.
\begin{theorem}\label{rolcpt}
Let $\mc{A}\in \mathbb{C}^{I_{1}\times\cdots\times I_{M} \times J_{1} \times \cdots \times J_{N} }$ and $\mc{B}\in \mathbb{C}^{J_{1} \times \cdots \times J_{N} \times K_{1}\times\cdots\times K_{L} }$.  Then  $(\mc{A}\n\mc{B})^{\dg} =\mc{B}^{\dg} \n \mc{A}^{\dg}$ if and only if  $\mc{A}^{\dg} \m\mc{A}$ commutes with $\mc{B}\lp\mc{B}^{H}$ and $\mc{A}^{H} \m\mc{A}$ commutes with $\mc{B}\lp\mc{B}^{\dg}$.

\end{theorem}
 Theorem 3.30, \cite{panigrahy2020} is reproduced here for any two tensors. 
\begin{theorem}\label{cp}
Let  $\mc{A}\in \mathbb{C}^{I_{1}\times\cdots\times I_{M} \times J_{1} \times \cdots\times J_{N} }$ and $\mc{B}\in \mathbb{C}^{J_{1} \times \cdots \times J_{N} \times K_{1}\times\cdots\times K_{L} }$. If $(\mc{A} \n \mc{B})^{\dg} = \mc{B}^{\dg} \n \mc{A}^{\dg} $, then  $\mc{A}^{\dg} \m \mc{A} $ 
and $\mc{B} \lp \mc{B}^{\dg} $ commute.
\end{theorem}
In general, the converse of the above result doesn't hold (see Example 3.31, \cite{panigrahy2020}).

\section{Main Results}
 This section is two-fold. 
First, we discuss some necessary and sufficient conditions for the reverse-order law of arbitrary order tensors. Some new properties of a $\{1,2,3\}$-inverse and a $\{1,2,4\}$-inverse of a tensor along with the reverse-order law are also presented. Second, we obtain a result in the direction of the additive Moore--Penrose inverse of tensors as an application of the reverse-order law, discuss some results on finding the additive perturbation bounds of the Moore–Penrose inverse and introduce sub-proper splitting of a tensor.
 \subsection{Reverse-order law of arbitrary order tensors}\label{secrolat}
In this subsection, we provide several necessary and sufficient conditions for the reverse-order law of arbitrary order tensors. Here onward, a tensor means it is of arbitrary order unless stated otherwise. The first main result of this subsection presents a necessary condition for the reverse-order law.
\begin{lemma}\label{mrolpT1}
Let $\mc{A}\in{\C}^{I_{1}\times \cdots\times I_{M}\times J_{1}\times\cdots\times J_{N}}$ and $\mc{B}\in{\C}^{J_{1}\times \cdots\times J_{N}\times K_{1}\times\cdots\times K_{L}}$. If $(\mc{A}\n\mc{B})^{\dg}=\mc{B}^{\dg}\n\mc{A}^{\dg}$, then 
$$(\mc{A}^{\dg}\m\mc{A}\n\mc{B}\lp\mc{B}^{\dg})^{\dg}=\mc{B}\lp\mc{B}^{\dg}\n\mc{A}^{\dg}\m\mc{A}.$$
\end{lemma}
\begin{proof}
Let $\mc{X}=\mc{A}^{\dg}\m\mc{A}$ and $\mc{Y}=\mc{B}\lp\mc{B}^{\dg}$. Then it is clear that $\mc{X}$ and $\mc{Y}$ are both Hermitian and idempotent, so $\mc{X}^{\dg}=\mc{X}=\mc{X}^{H}$ and $\mc{Y}^{\dg}=\mc{Y}= \mc{Y}^{H}$. Thus,  $\mc{X}^{\dg}\n\mc{X}$ commutes with $\mc{Y}\n\mc{Y}^{H}$ by Theorem  \ref{cp}.
So, the claim follows by Theorem \ref{rolcpt}.
\end{proof}
We provide below an example which shows that the converse of the above result is not true. 
\begin{example}
Let $\mc{A}=(a_{ijk})\in{\C}^{2\times3\times4}$ and $\mc{B}=(b_{ij})\in{\C}^{4\times 2}$ be such that
\begin{center}
 \begin{tabular}{c c c | c c c | c c c | c c c }
\hline
  \multicolumn{3}{c}{$\mc{A}(:,:,1)$}   & \multicolumn{3}{c}{$\mc{A}(:,:,2)$} & \multicolumn{3}{c}{$\mc{A}(:,:,3)$} & \multicolumn{3}{c}{$\mc{A}(:,:,4)$} \\
\hline
$1$ & $0$ & $1$ & $0$ & $1$ & $0$ & $1$ & $0$ & $0$ & $0$ & $0$ & $0$\\
$0$ & $1$ & $0$ & $0$ & $0$ & $1$ & $0$ & $0$ & $1$ & $0$ & $1$ & $0$\\
 \hline
\end{tabular},
and $\mc{B}(:,:)=\begin{pmatrix}1&0\\0&1\\1&0\\0&1\end{pmatrix}$.
\end{center}
Then  $\mc{A}\1\mc{B}=(c_{ijk})\in{\C}^{2\times3\times2}$, where
 \begin{center}
 \begin{tabular}{c c c | c c c }
\hline
  \multicolumn{3}{c}{$\mc{A}\1\mc{B}(:,:,1)$}   & \multicolumn{3}{c}{$\mc{A}\1\mc{B}(:,:,2)$} \\
\hline
$2$ & $0$ & $1$ & $0$ & $1$ & $0$\\
$0$ & $1$ & $1$ & $0$ & $1$ & $1$\\
 \hline
\end{tabular}.
\end{center}
The Moore--Penrose inverses $\mc{A}^{\dg}=(a'_{ijk})\in{\C}^{4\times2\times3}$, $\mc{B}^{\dg}=(b'_{ij})\in{\C}^{2\times 4}$ and $(\mc{A}\1\mc{B})^{\dg}=(c'_{ij})\in{\C}^{2\times 2\times 3}$, are calculated as follows:
\begin{center}
 \begin{tabular}{c c | c c | c c }
\hline
  \multicolumn{2}{c}{$\mc{A}^{\dg}(:,:,1)$}   & \multicolumn{2}{c}{$\mc{A}^{\dg}(:,:,2)$} & \multicolumn{2}{c}{$\mc{A}^{\dg}(:,:,3)$} \\
\hline
$1/4$ & $0$ & $1/4$ & $0$ & $3/4$ & $-1/4$\\
$-1/4$ & $0$ & $3/4$ & $0$ & $1/4$ & $1/4$\\
$1/2$ & $0$ & $-1/2$ & $0$ & $-1/2$ & $1/2$\\
$-1/4$ & $0$ & $-1/4$ & $1$ & $-3/4$ & $1/4$\\
 \hline
\end{tabular},
 $\mc{B}^{\dg}(:,:)=\begin{pmatrix}1/2 & 0 & 1/2 & 0 \\ 0 & 1/2 & 0 & 1/ 2 \end{pmatrix}$,
\end{center}
and
\begin{center}
    \begin{tabular}{c c | c c | c c }
\hline
  \multicolumn{2}{c}{$(\mc{A}\1\mc{B})^{\dg}(:,:,1)$}   & \multicolumn{2}{c}{$(\mc{A}\1\mc{B})^{\dg}(:,:,2)$} & \multicolumn{2}{c}{$(\mc{A}\1\mc{B})^{\dg}(:,:,3)$}\\
\hline
$6/17$ & $0$ & $-2/17$ & $1/17$ & $3/17$ & $1/17$\\
$-4/17$ & $0$ & $7/17$ & $5/17$ & $-2/17$ & $5/17$\\
\hline
\end{tabular}.
\end{center}
So, we get $\mc{B}^{\dg}\1\mc{A}^{\dg}\in{\C}^{2\times 2\times 3}$, where
\begin{center}
    \begin{tabular}{c c | c c | c c }
\hline
  \multicolumn{2}{c}{$\mc{B}^{\dg}\1\mc{A}^{\dg}(:,:,1)$} & \multicolumn{2}{c}{$\mc{B}^{\dg}\1\mc{A}^{\dg}(:,:,2)$} & \multicolumn{2}{c}{$\mc{B}^{\dg}\1\mc{A}^{\dg}(:,:,3)$} \\
\hline
$3/8$ & $0$ & $-1/8$ & $0$ & $1/8$ & $1/8$\\
$-1/4$ & $0$ & $1/4$ & $1/2$ & $-1/4$ & $1/4$\\
\hline
\end{tabular}.
\end{center}
We thus have $$\mc{B}\1\mc{B}^{\dg}\1\mc{A}^{\dg}\2\mc{A}(:,:)=\begin{pmatrix}1\slash2&0&1\slash2&0\\0&1\slash2&0&1\slash2 \\1\slash2&0&1\slash2&0\\0&1\slash2&0&1\slash2 \end{pmatrix},~
\mc{A}^{\dg}\2\mc{A}\1\mc{B}\1\mc{B}^{\dg}(:,:)=\begin{pmatrix}1\slash2&0&1\slash2&0\\0&1\slash2&0&1\slash2 \\1\slash2&0&1\slash2&0\\0&1\slash2&0&1\slash2 \end{pmatrix},$$
and $$(\mc{A}^{\dg}\2\mc{A}\1\mc{B}\1\mc{B}^{\dg})^{\dg}(:,:)=\begin{pmatrix}1\slash2&0&1\slash2&0\\0&1\slash2&0&1\slash2 \\1\slash2&0&1\slash2&0\\0&1\slash2&0&1\slash2 \end{pmatrix}.$$
Whence, $(\mc{A}^{\dg}\2\mc{A}\1\mc{B}\1\mc{B}^{\dg})^{\dg}=\mc{B}\1\mc{B}^{\dg}\1\mc{A}^{\dg}\2\mc{A}$, but $(\mc{A}\1\mc{B})^{\dg}\neq \mc{B}^{\dg}\1\mc{A}^{\dg}$.
\end{example}
The converse of Lemma \ref{mrolpT1} is true under the assumption of the fact that $\mc{A}\n\mc{B}\lp\mc{B}^{\dg}\n\mc{A}^{\dg}$ and $\mc{B}^{\dg}\n\mc{A}^{\dg}\m\mc{A}\n\mc{B}$ are Hermitian. The stated fact is proved below. 
\begin{lemma}\label{Th2}
Let $\mc{A}\in{\C}^{I_{1}\times \cdots\times I_{M}\times J_{1}\times\cdots\times J_{N}}$ and $\mc{B}\in{\C}^{J_{1}\times \cdots\times J_{N}\times K_{1}\times\cdots\times K_{L}}$. Also, let $\mc{P}=\mc{A}^{\dg}\m\mc{A}$ and $\mc{Q}=\mc{B}\lp\mc{B}^{\dg}$. If $\mc{P}\n\mc{Q}=\mc{Q}\n\mc{P}$, $(\mc{A}\n\mc{B}\lp\mc{B}^{\dg}\n\mc{A}^{\dg})^{H}=\mc{A}\n\mc{B}\lp\mc{B}^{\dg}\n\mc{A}^{\dg}$ and $(\mc{B}^{\dg}\n\mc{A}^{\dg}\m\mc{A}\n\mc{B})^{H}= \mc{B}^{\dg}\n\mc{A}^{\dg}\m\mc{A}\n\mc{B}$,
then $(\mc{A}\n\mc{B})^{\dg}=\mc{B}^{\dg}\n\mc{A}^{\dg}$.
\end{lemma}
\begin{proof}
Since $\mc{P}\n\mc{Q}=\mc{Q}\n\mc{P}$, so by Equation \eqref{mpeq1} we get
\begin{equation}\label{rolse1}
 \mc{A}^{\dg}\m\mc{A}\n\mc{B}\lp\mc{B}^{\dg}\n \mc{A}^{\dg}\m\mc{A}\n\mc{B}\lp\mc{B}^{\dg} = \mc{A}^{\dg}\m\mc{A}\n\mc{B}\lp\mc{B}^{\dg},
\end{equation}
and by Equation \eqref{mpeq2} we get
\begin{equation}\label{rolse2}
 \mc{B}\lp\mc{B}^{\dg}\n\mc{A}^{\dg}\m\mc{A}\n\mc{B}\lp\mc{B}^{\dg}\n\mc{A}^{\dg}\m\mc{A} = \mc{B}\lp\mc{B}^{\dg}\n\mc{A}^{\dg}\m\mc{A}.
\end{equation}
Let $\mc{X}=\mc{A}\n\mc{B}$ and $\mc{Y}=\mc{B}^{\dg}\n\mc{A}^{\dg}$. Pre-multiplying $\mc{A}$ and post-multiplying $\mc{B}$ to Equation \eqref{rolse1} yields $\mc{X}\lp\mc{Y}\m\mc{X} = \mc{X}$, and on pre-multiplication of $\mc{B}^{\dg}$ and post-multiplication of $\mc{A}^{\dg}$ with Equation \eqref{rolse2} results $\mc{Y}\m\mc{X}\lp\mc{Y}=\mc{Y}$. The Hermitian property of $\mc{X}\lp\mc{Y}$ and $\mc{Y}\m\mc{X}$ are confirmed by the second and third assumptions. The claim is thus attained by Definition \ref{defmpi}.
\end{proof}

Theorem \ref{cp} and Lemma \ref{Th2} can now be together stated as following.
\begin{theorem}
Let $\mc{A}\in{\C}^{I_{1}\times \cdots\times I_{M}\times J_{1}\times\cdots\times J_{N}}$ and $\mc{B}\in{\C}^{J_{1}\times \cdots\times J_{N}\times K_{1}\times\cdots\times K_{L}}$. Then  $(\mc{A}\n\mc{B})^{\dg}=\mc{B}^{\dg}\n\mc{A}^{\dg}$ if and only if 
\begin{itemize}
    \item[(i)] $\mc{A}^{\dg}\m\mc{A}\n\mc{B}\lp\mc{B}^{\dg}=\mc{B}\lp\mc{B}^{\dg}\n\mc{A}^{\dg}\m\mc{A}$,
    \item[(ii)] $(\mc{A}\n\mc{B}\lp\mc{B}^{\dg}\n\mc{A}^{\dg})^{H}=\mc{A}\n\mc{B}\lp\mc{B}^{\dg}\n\mc{A}^{\dg}$,
    \item[(iii)] $(\mc{B}^{\dg}\n\mc{A}^{\dg}\m\mc{A}\n\mc{B})^{H}=\mc{B}^{\dg}\n\mc{A}^{\dg}\m\mc{A}\n\mc{B}$.
\end{itemize}
\end{theorem}
Now, we move to the proof an interesting result which is helpful to improve the reverse-order law.
\begin{lemma}\label{rol124}
Let $\mc{A}\in{\C}^{I_{1}\times \cdots\times I_{M}\times J_{1}\times\cdots\times J_{N}}$ and $\mc{B}\in{\C}^{J_{1}\times \cdots\times J_{N}\times K_{1}\times\cdots\times K_{L}}$. If $\mc{A}^{\dg}\m\mc{A}\n\mc{B}\lp\mc{B}^{H}\n\mc{A}^{H}=\mc{B}\lp\mc{B}^{H}\n\mc{A}^{H}$, then $\mc{B}^{\dg}\n\mc{A}^{\dg}\in (\mc{A}\n\mc{B})\{1,2,4\}$.
\end{lemma}
\begin{proof}
Pre-multiplying and post-multiplying $\mc{A}^{\dg}\m\mc{A}\n\mc{B}\lp\mc{B}^{H}\n\mc{A}^{H}=\mc{B}\lp\mc{B}^{H}\n\mc{A}^{H}$ by $\mc{B}^{\dg}$ and $((\mc{A}\n\mc{B})^{H})^{\dg}$, respectively, we get
\begin{eqnarray*}
\mc{B}^{\dg}\n\mc{A}^{\dg}\m\mc{A}\n\mc{B}\lp\mc{B}^{H}\n\mc{A}^{H}\m((\mc{A}\n\mc{B})^{H})^{\dg}&=&\mc{B}^{\dg}\n \mc{B}\lp\mc{B}^{H}\n\mc{A}^{H}  \m((\mc{A}\n\mc{B})^{H})^{\dg}\\
&=&(\mc{A}\n\mc{B})^{\dg}\m\mc{A}\n\mc{B}.
\end{eqnarray*}
But 
\begin{eqnarray*}
\mc{B}^{\dg}\n\mc{A}^{\dg}\m\mc{A}\n\mc{B}\lp\mc{B}^{H}\n\mc{A}^{H}\m((\mc{A}\n\mc{B})^{H})^{\dg}&=& \mc{B}^{\dg}\n\mc{A}^{\dg}\m\mc{A}\n\mc{B}\lp(\mc{A}\n\mc{B})^{H}\m((\mc{A}\n\mc{B})^{\dg})^{H}\\
&=& \mc{B}^{\dg}\n\mc{A}^{\dg}\m\mc{A}\n\mc{B}.
\end{eqnarray*}
Thus, we have 
\begin{equation}\label{eq21}
   \mc{B}^{\dg} \n\mc{A}^{\dg} \m\mc{A} \n\mc{B} = (\mc{A}\n\mc{B})^{\dg} \m \mc{A}\n\mc{B}.
\end{equation}
 Let $\mc{X}=\mc{A}\n\mc{B}$ and $\mc{Y}=\mc{B}^{\dg}\n\mc{A}^{\dg}$. Then pre-multiplying $\mc{A}\n\mc{B}$ to  Equation \eqref{eq21}, we get $\mc{X}\lp\mc{Y}\m\mc{X}=\mc{X}$. Using Lemma \ref{lm3.13}, we have 
 \begin{equation*}
     \mc{B}\lp\mc{B}^{H}\n\mc{A}^{\dg}\m\mc{A}\n\mc{B}\lp\mc{B}^{\dg} \n\mc{A}^{H} = \mc{B}\lp\mc{B}^{H} \n \mc{A}^{H}.
 \end{equation*}
Pre-multiplying $\mc{B}^{\dg}$ to the above equation, we obtain
 \begin{equation*}
     \mc{B}^{H}\n\mc{A}^{\dg}\m\mc{A}\n\mc{B}\lp\mc{B}^{\dg} \n\mc{A}^{H} = \mc{B}^{H} \n \mc{A}^{H},
 \end{equation*}
 which reduces to 
 \begin{equation*}
     \mc{B}\lp\mc{B}^{\dg}\n\mc{A}^{\dg}\m\mc{A}\n\mc{B}\lp\mc{B}^{\dg} \n\mc{A}^{\dg}\m\mc{A} = \mc{B}\lp\mc{B}^{\dg} \n \mc{A}^{\dg}\m\mc{A}
 \end{equation*}
 by pre-multiplying $\mc{B}^{\dg H}$ and post-multiplying $\mc{A}^{\dg H}$ again, simultaneously. The above equation gives $\mc{Y}\m\mc{X}\lp\mc{Y}=\mc{Y}$ by pre-multiplying $\mc{B}^{\dg}$ and post-multiplying $\mc{A}^{\dg}$. The Hermitian property of $\mc{Y}\m\mc{X}$ can be proved by using Equation \eqref{eq21}.
 Thus, $\mc{Y}\in \mc{X}\{1,2,4\}$, i.e., $\mc{B}^{\dg}\n\mc{A}^{\dg}\in (\mc{A}\n\mc{B})\{1,2,4\}$ by Definition \ref{defmpi}.
\end{proof}
It can also be proved that if $\mc{B}\lp\mc{B}^{\dg}\n\mc{A}^{H}\m\mc{A}\n\mc{B}=\mc{A}^{H}\m\mc{A}\n\mc{B}$, then $\mc{B}^{\dg}\n\mc{A}^{\dg}\in (\mc{A}\n\mc{B})\{1,2,3\}$.
Based on the above fact, Theorem \ref{rev1} can now be improved, and the improved version is presented next. 
\begin{theorem}\label{ivrev1}
Let $\mc{A}\in{\C}^{I_{1}\times \cdots\times I_{M}\times J_{1}\times\cdots\times J_{N}}$ and $\mc{B}\in{\C}^{J_{1}\times \cdots\times J_{N}\times K_{1}\times\cdots\times K_{L}}$. Then  $(\mc{A}\n\mc{B})^{\dg}=\mc{B}^{\dg}\n\mc{A}^{\dg}$ if and only if 
\begin{itemize}
    \item[(i)] $\mc{A}^{\dg}\m\mc{A}\n\mc{B}\lp\mc{B}^{H}\n\mc{A}^{H}=\mc{B}\lp\mc{B}^{H}\n\mc{A}^{H}$,
    \item[(ii)] $\mc{A}\n\mc{B}\lp\mc{B}^{\dg}\n\mc{A}^{\dg}$ is Hermitian.
\end{itemize}
\end{theorem}
 
\begin{proof}
Let $ (\mc{A}\n\mc{B})^{\dg}  =  \mc{B}^{\dg} \n\mc{A}^{\dg}$. So $\mc{A}\n\mc{B}\lp\mc{B}^{\dg}\n\mc{A}^{\dg}=\mc{A}\n\mc{B}\lp(\mc{A}\n\mc{B})^{\dg}$ is Hermitian by the definition of the Moore--Penrose inverse of a tensor. By Theorem \ref{rolcpt}, we have $(i)$.

Conversely,   $(i)$ implies $\mc{B}^{\dg}\n\mc{A}^{\dg}\in (\mc{A}\n\mc{B})\{1,2,4\}$ by Lemma \ref{rol124} and in addition with condition $(ii)$ the claim is justified.
\end{proof}
 Now, another improved version of Theorem \ref{rev1} can be stated as  $(\mc{A}\n\mc{B})^{\dg}=\mc{B}^{\dg}\n\mc{A}^{\dg}$  if and only if 
 $\mc{B}\lp\mc{B}^{\dg}\n\mc{A}^{H}\m\mc{A}\n\mc{B}=\mc{A}^{H}\m\mc{A}\n\mc{B}$, and $\mc{B}^{\dg}\n\mc{A}^{\dg}\m\mc{A}\n\mc{B}$ is Hermitian.
A new characterization of the reverse-order law is proved next. 
\begin{theorem}
Let $\mc{A}\in{\C}^{I_{1}\times \cdots\times I_{M}\times J_{1}\times\cdots\times J_{N}}$ and $\mc{B}\in{\C}^{J_{1}\times \cdots\times J_{N}\times K_{1}\times\cdots\times K_{L}}$. Then  $(\mc{A}\n\mc{B})^{\dg}=\mc{B}^{\dg}\n\mc{A}^{\dg}$  is equivalent to 
\begin{itemize}
    \item[(i)] $(\mc{A}^{\dg}\m\mc{A}\n\mc{B})^{\dg}=\mc{B}^{\dg}\n\mc{A}^{\dg}\m\mc{A}$,
    \item[(ii)] $(\mc{A}\n\mc{B})^{\dg}=(\mc{A}^{\dg}\m\mc{A}\n\mc{B})^{\dg}\n\mc{A}^{\dg}$.
\end{itemize}
\end{theorem}

\begin{lemma}\label{rpsntfnorm}
Let~  $\mc{U}=\begin{bmatrix}\mc{U}_{1}&\mc{U}_{2}\end{bmatrix}\in{\C}^{I_{1}\times \cdots\times I_{M}\times I_{1}\times \cdots\times I_{M}}$ and $\mc{V}=\begin{bmatrix}\mc{V}_{1}&\mc{V}_{2}\end{bmatrix}\in{\C}^{J_{1}\times \cdots\times J_{N}\times J_{1}\times \cdots\times J_{N}}$ be two unitary tensors. Then, 
\begin{equation*}
    \|\mc{E}\|_{F}^{2}=\|\mc{U}_{1}^{H}\m\mc{E}\n\mc{V}_{1}\|_{F}^{2}+\|\mc{U}_{1}^{H}\m\mc{E}\n\mc{V}_{2}\|_{F}^{2}+\|\mc{U}_{2}^{H}\m\mc{E}\n\mc{V}_{1}\|_{F}^{2}+\|\mc{U}_{2}^{H}\m\mc{E}\n\mc{V}_{2}\|_{F}^{2},
\end{equation*}
for any tensor $\mc{E}\in{\C}^{I_{1}\times \cdots\times I_{M}\times J_{1}\times \cdots \times J_{N}}$.
\end{lemma}
\begin{proof}
Since the Frobenius norm is unitarily invariant, so we have
\begin{eqnarray*}
\|\mc{E}\|_{F}^{2}&=&\|\mc{U}^{H}\m\mc{E}\n\mc{V}\|_{F}^{2}\\
&=&\left\|\begin{bmatrix}U_{1}^{H}\\\mc{U}_{2}^{H}\end{bmatrix}\m\mc{E}\n\begin{bmatrix}\mc{V}_{1}&\mc{V}_{2}\end{bmatrix}\right\|_{F}^{2}\\
&=&\left\|\begin{bmatrix}\mc{U}_{1}^{H}\m\mc{E}\n\mc{V}_{1}&\mc{U}_{1}^{H}\m\mc{E}\n\mc{V}_{2}\\
\mc{U}_{2}^{H}\m\mc{E}\n\mc{V}_{1}&\mc{U}_{2}^{H}\m\mc{E}\n\mc{V}_{2}\end{bmatrix}\right\|_{F}^{2}\\
&=&\|\mc{U}_{1}^{H}\m\mc{E}\n\mc{V}_{1}\|_{F}^{2}+\|\mc{U}_{1}^{H}\m\mc{E}\n\mc{V}_{2}\|_{F}^{2}+\|\mc{U}_{2}^{H}\m\mc{E}\n\mc{V}_{1}\|_{F}^{2}+\|\mc{U}_{2}^{H}\m\mc{E}\n\mc{V}_{2}\|_{F}^{2}.
\end{eqnarray*}
\end{proof}
Let $\mc{A}, ~\mc{E}\in{\C}^{I_{1}\times \cdots\times I_{N}\times J_{1}\times \cdots \times J_{N}}$ and $\mc{B}=\mc{A}+\mc{E}$. Let $\mc{A}=\mc{U}\m\mc{D}\n\mc{V}^{H}$ and $\mc{B}=\mc{R}\m\mc{S}\n\mc{T}^{H}$ be the singular value decompositions (SVD) of $\mc{A}$ and $\mc{B}$, where $\mc{U},~\mc{R}\in{\C}^{I_{1}\times\cdots\times I_{N}\times I_{1}\times \cdots\times I_{N}}$ and $\mc{V},~\mc{T}\in{\C}^{J_{1}\times\cdots\times J_{N}\times J_{1}\times\cdots\times J_{N}}$ satisfy 
\begin{equation*}
    \mc{U}\n\mc{U}^{H}=\mc{I}_{N},~\mc{V}\n\mc{V}=\mc{I}_{N},~\mc{R}\n\mc{R}^{H}=\mc{I}_{N},~\mc{T}\n\mc{T}^{H}=\mc{I}_{N}.
\end{equation*}
Suppose that $rank_{U}(\mc{A})=r$ and $rank_{U}(\mc{B})=s$. Let us arrange the tensors involving in the SVD of $\mc{A}$ and $\mc{B}$ as the following block tensors.
\begin{eqnarray*}
    \mc{U}=\begin{bmatrix}\mc{U}_{1} &\mc{U}_{2}\end{bmatrix},~&
    \mc{D}=\begin{bmatrix}\mc{D}_{1} &\mc{O}\\
    \mc{O}&\mc{O}\end{bmatrix},& \mc{V}=\begin{bmatrix}\mc{V}_{1} &\mc{V}_{2}\end{bmatrix};\\
    \mc{R}=\begin{bmatrix}\mc{R}_{1} &\mc{R}_{2}\end{bmatrix},&  \mc{S}=\begin{bmatrix}\mc{S}_{1}~&\mc{O}\\
    \mc{O}&\mc{O}\end{bmatrix},& \mc{T}=\begin{bmatrix}\mc{T}_{1} &\mc{T}_{2}\end{bmatrix},
\end{eqnarray*}
where 
$\mc{U}_{1}\in{\C}^{I_{1}\times\cdots\times I_{N}\times K_{1}\times \cdots\times K_{N}}, ~\mc{U}_{2}\in{\C}^{I_{1}\times\cdots\times I_{N}\times L_{1}\times \cdots\times L_{M}}$, such that $K_{i}+L_{i}=I_{i}$ for $i=1,~2,~\cdots,~N$,  
$\mc{V}_{1}\in{\C}^{J_{1}\times \cdots\times J_{N}\times K_{1}\times \cdots\times K_{N}}, ~\mc{V}_{2}\in{\C}^{J_{1}\times \cdots\times J_{N}\times Q_{1}\times \cdots\times Q_{N}}$ such that $K_{i}+Q_{i}=J_{i}$ for $i=1,~2,~\cdots,~N$;
$\mc{R}_{1}\in{\C}^{I_{1}\times\cdots\times I_{N}\times \Tilde{K}_{1}\times \cdots\times \Tilde{K}_{N}}, ~\mc{R}_{2}\in{\C}^{I_{1}\times\cdots\times I_{N}\times \Tilde{L}_{1}\times \cdots\times \Tilde{L}_{N}}$, such that $\Tilde{K}_{i}+\Tilde{L}_{i}=I_{i}$ for $i=1,~2,~\cdots,~N$,
$\mc{T}_{1}\in{\C}^{J_{1}\times \cdots\times J_{N}\times \Tilde{K}_{1}\times \cdots\times \Tilde{K}_{N}}, ~\mc{T}_{2}\in{\C}^{J_{1}\times \cdots\times J_{N}\times \Tilde{Q}_{1}\times \cdots\times \Tilde{Q}_{N}}$ such that $\Tilde{K}_{i}+\Tilde{Q}_{i}=J_{i}$ for $i=1,~2,~\cdots,~N$. Here $\mc{D}_{1}\in{\C}^{K_{1}\times \cdots\times K_{N}\times K_{1}\times \cdots \times K_{N}}$ and $\mc{S}_{1}\in{\C}^{\Tilde{K}_{1}\times \cdots\times \Tilde{K}_{N}\times \Tilde{K}_{1}\times \cdots\times \Tilde{K}_{N}}$ with $K_{1}\cdot \ldots\cdot K_{N}=r$ and $\Tilde{K}_{1}\cdot\ldots\cdot \Tilde{K}_{N}=s$ are diagonal tensors defined element-wise as follow
\begin{equation*}
(\mc{D}_{1})_{k_{1}\cdots k_{N}j_{1}\cdots j_{N}}=\begin{cases}
    \sigma_{ivec{(\mathbf{k},\mathbb{K})}},&\text{if } ivec{(\mathbf{k},\mathbb{K})}=ivec{(\mathbf{j},\mathbb{J})},\\
    0,&\text{ otherwise},
    \end{cases}
\end{equation*}
and
\begin{equation*}
    (\mc{S}_{1})_{\Tilde{k}_{1}\cdots \Tilde{k}_{N}\Tilde{j}_{1}\cdots \Tilde{j}_{N}}=\begin{cases}
    \delta_{ivec{(\mathbf{\Tilde{k}},\Tilde{\mathbb{K}})}},&\text{if } ivec{(\mathbf{\Tilde{k}},\Tilde{\mathbb{K}})}=ivec{(\Tilde{\mathbf{j}},\Tilde{\mathbb{J}})},\\
    0,&\text{ otherwise},
    \end{cases}
\end{equation*}
where $\mathbf{k}:=\{k_{1},~\cdots,~k_{N}\}$, $\mathbf{j}:=\{j_{1},~\cdots,~j_{N}\}$, $\mathbf{\Tilde{k}}:=\{\Tilde{k}_{1},~\cdots,~\Tilde{k}_{M}\}$,
$\Tilde{\mathbf{j}}:=\{\Tilde{j}_{1},~\cdots,~\Tilde{j}_{N}\}$,
$\mathbb{K}:=\{K_{1},~\cdots,~K_{N}\}$, $\mathbb{J}:=\{K_{1},~\cdots,~K_{N}\}$,
$\Tilde{\mathbb{K}}:=\{\Tilde{K}_{1},~\cdots,~\Tilde{K}_{M}\}$ and $\Tilde{\mathbb{J}}:=\{\Tilde{K}_{1},~\cdots,~\Tilde{K}_{M}\}$; $\sigma_{1}\geq \sigma_{2}\geq\cdots\geq\sigma_{r}> 0$ and $\delta_{1}\geq \delta_{2}\geq\cdots\geq\delta_{s}> 0$ are singular values of $\mc{A}$ and $\mc{B}$, respectively.
By Proposition 2.4 \cite{sun2016}, we  then have
\begin{eqnarray}
    \mc{A}
    &=& \mc{U}_{1}\n\mc{D}_{1}\n\mc{V}_{1}^{H},
\end{eqnarray}
and
\begin{eqnarray}
    \mc{B}
    &=& \mc{R}_{1}\n\mc{S}_{1}\n\mc{T}_{1}^{H}.
\end{eqnarray}
So,
\begin{eqnarray}
    \mc{E}&=&\mc{B}-\mc{A}\nonumber\\
    &=& {R}_{1}\n\mc{S}_{1}\n\mc{T}_{1}^{H}-\mc{U}_{1}\n\mc{D}_{1}\n\mc{V}_{1}^{H}.
\end{eqnarray}
Thus,
\begin{eqnarray}
    \mc{R}_{1}^{H}\n\mc{E}\n\mc{V}_{1}
    &=& \mc{S}_{1}\n\mc{T}^{H}\n\mc{V}_{1}-\mc{R}_{1}^{H}\n\mc{U}_{1}\n\mc{D}_{1},\label{eqrhev}
\end{eqnarray}
and
\begin{eqnarray}
    \mc{U}_{1}^{H}\n\mc{E}\n\mc{T}_{1}
    &=&\mc{U}_{1}^{H}\n\mc{R}_{1}\n\mc{S}_{1}-\mc{D}_{1}\n\mc{V}_{1}^{H}\n\mc{T}_{1}.\label{eqrhev2}
\end{eqnarray}
From Equations \eqref{eqrhev} and \eqref{eqrhev2}, we now obtain
\begin{eqnarray}\label{altfrm}
   \mc{S}_{1}^{-1}\n\mc{R}_{1}^{H}\n\mc{E}\n\mc{V}_{1}\n\mc{D}_{1}^{-1}
   &=&\mc{T}^{H}\n\mc{V}_{1}\n\mc{D}_{1}^{-1}-\mc{S}_{1}^{-1}\n\mc{R}_{1}^{H}\n\mc{U}_{1},
\end{eqnarray}
and
\begin{eqnarray}
    \mc{D}_{1}^{-1}\n\mc{U}_{1}^{H}\n\mc{E}\n\mc{T}_{1}\n\mc{S}_{1}^{-1}
    &=&\mc{D}_{1}^{-1}\n\mc{U}_{1}^{H}\n\mc{R}_{1}-\mc{V}_{1}^{H}\n\mc{T}_{1}\n\mc{S}_{1}^{-1}.
\end{eqnarray}
Since $\mc{U}_{2}^{H}\n\mc{U}_{1}=\mc{O}$ and $\mc{T}_{1}^{H}\n\mc{T}_{2}=\mc{O}$, we get
\begin{eqnarray}
    \mc{U}_{2}^{H}\n\mc{E}\n\mc{T}_{1}
    &=&\mc{U}_{2}^{H}\n\mc{R}_{1}\n\mc{S}_{1},
\end{eqnarray}
and
\begin{eqnarray}
   \mc{U}_{1}^{H}\n\mc{E}\n\mc{T}_{2}
   &=&-\mc{D}_{1}\n\mc{V}_{1}^{H}\n\mc{T}_{2}.
\end{eqnarray}
Since $\mc{V}_{1}^{H}\n\mc{V}_{2}=\mc{O}$ and $\mc{R}_{2}^{H}\n\mc{R}_{1}=\mc{O}$, so
\begin{eqnarray}
    \mc{R}_{1}^{H}\n\mc{E}\n\mc{V}_{2}
    &=&\mc{S}_{1}\n\mc{T}_{1}^{H}\n\mc{V}_{2},
\end{eqnarray}
and
\begin{eqnarray}
    \mc{R}_{2}^{H}\n\mc{E}\n\mc{V}_{1}
    &=&-\mc{R}_{2}^{H}\n\mc{U}_{1}\n\mc{D}_{1}.
\end{eqnarray}
Also, we have
\begin{eqnarray}
   \mc{S}_{1}^{-1}\n\mc{R}_{1}^{H}\n\mc{U}_{2}
   &=&\mc{S}_{1}^{-2}\n(\mc{U}_{2}^{H}\n\mc{R}_{1}\n\mc{S}_{1})^{H}\nonumber\\
   &=&\mc{S}_{1}^{-2}\n(\mc{U}_{2}^{H}\n\mc{E}\n\mc{T}_{1})^{H},\label{eqveqn1}
\end{eqnarray}
and
\begin{eqnarray}
   -\mc{T}_{2}^{H}\n\mc{V}_{1}\n\mc{D}_{1}^{-1}
   &=&(-\mc{D}_{1}\n\mc{V}_{1}^{H}\n\mc{T}_{2})^{H}\n\mc{D}_{1}^{-2}\nonumber\\
   &=&(\mc{U}_{1}^{H}\n\mc{E}\n\mc{T}_{2})^{H}\n\mc{D}_{1}^{-2}.\label{eqveqn2}
\end{eqnarray}
The additive perturbation bounds for the Moore–Penrose inverse of arbitrary order tensors via the Einstein product by means of the Frobenius norm is obtained next.

\begin{theorem}
Let $\mc{A},~\mc{E}\in{\C}^{I_{1}\times\cdots\times I_{N}\times J_{1}\times\cdots\times J_{N}}$ and $\mc{B}=\mc{A}+\mc{E}$. If $rank_{U}(\mc{A})=r$ and $rank_{U}(\mc{B})=s$, then 
\begin{equation}
    \|\mc{B}^{\dg}-\mc{A}^{\dg}\|_{F}\leq
    \max\left\{\|\mc{A}^{\dg}\|_{2}\|\mc{B}^{\dg}\|_{2},~\|\mc{A}^{\dg}\|_{2}^{2},~\|\mc{B}^{\dg}\|_{2}^{2}\right\}\|\mc{E}\|_{F}.
\end{equation}
\begin{proof}
Since $\mc{A}^{\dg}=\mc{V}_{1}\n\mc{D}_{1}^{-1}\n\mc{U}_{1}^{H}$ and $\mc{B}^{\dg}=\mc{T}_{1}\n\mc{S}_{1}^{-1}\n\mc{R}_{1}^{H}$, so
\begin{eqnarray}
   \|\mc{B}^{\dg}-\mc{A}^{\dg}\|_{F}^{2}&=&\|\mc{T}_{1}\n\mc{S}_{1}^{-1}\n\mc{R}_{1}^{H}-\mc{V}_{1}\n\mc{D}_{1}^{-1}\n\mc{U}_{1}^{H}\|_{F}^{2}\nonumber\\
   &=&\|\mc{T}^{H}\n\left(\mc{T}_{1}\n\mc{S}_{1}^{-1}\n\mc{R}_{1}^{H}-\mc{V}_{1}\n\mc{D}_{1}^{-1}\n\mc{U}_{1}^{H}\right)\n\mc{U}\|_{F}^{2}\nonumber\\
   &=&\left\|\begin{bmatrix}\mc{T}_{1}^{H}\\ \mc{T}_{2}^{H}\end{bmatrix}\n\left(\mc{T}_{1}\n\mc{S}_{1}^{-1}\n\mc{R}_{1}^{H}-\mc{V}_{1}\n\mc{D}_{1}^{-1}\n\mc{U}_{1}^{H}\right)\n\begin{bmatrix}\mc{U}_{1} &\mc{U}_{2}\end{bmatrix}\right\|_{F}^{2}\nonumber\\
   &=&\left\|\begin{bmatrix}\mc{S}_{1}^{-1}\n\mc{R}_{1}^{H}\n\mc{U}_{1}-\mc{T}_{1}^{H}\n\mc{V}_{1}\n\mc{D}_{1}^{-1}&\mc{S}_{1}^{-1}\n\mc{R}_{1}^{H}\n\mc{U}_{2}\\
   -\mc{T}_{2}^{H}\n\mc{V}_{1}\n\mc{D}_{1}^{-1}&\mc{O}\end{bmatrix}\right\|_{F}^{2}\nonumber\\
   &=&\left\|\begin{bmatrix}\mc{S}_{1}^{-1}\n\mc{R}_{1}^{H}\n\mc{E}\n\mc{V}_{1}\n\mc{D}_{1}^{-1}&\mc{S}_{1}^{-2}\n\mc{T}_{1}^{H}\n\mc{E}^{H}\n\mc{U}_{2}\\
   \mc{T}_{2}^{H}\n\mc{E}^{H}\n\mc{U}_{1}\n\mc{D}^{-2}&\mc{O}\end{bmatrix}\right\|_{F}^{2}.\nonumber
\end{eqnarray}
On using properties of Frobenius norm, we obtain
\begin{eqnarray}
&&\|\mc{B}^{\dg}-\mc{A}^{\dg}\|_{F}^{2}\nonumber\\
&=& \|\mc{S}_{1}^{-1}\n\mc{R}_{1}^{H}\n\mc{E}\n\mc{V}_{1}\n\mc{D}_{1}^{-1}\|_{F}^{2}+\|\mc{S}_{1}^{-2}\n\mc{T}_{1}^{H}\n\mc{E}^{H}\n\mc{U}_{2}\|_{F}^{2}+\|\mc{T}_{2}^{H}\n\mc{E}^{H}\n\mc{U}_{1}\n\mc{D}^{-2}\|_{F}^{2}\nonumber\\
    &\leq&\dfrac{1}{\sigma_{r}^{2}\delta_{s}^{2}}\|\mc{R}_{1}^{H}\n\mc{E}\n\mc{V}_{1}\|_{F}^{2}+\dfrac{1}{\delta_{s}^{4}}\|\mc{T}_{1}^{H}\n\mc{E}^{H}\n\mc{U}_{2}\|_{F}^{2}+\dfrac{1}{\sigma_{r}^{4}}\|\mc{T}_{2}^{H}\n\mc{E}^{H}\n\mc{U}_{1}\|_{F}^{2}\nonumber\\
   &\leq&\max\left\{\dfrac{1}{\sigma_{r}^{2}\delta_{s}^{2}},\dfrac{1}{\delta_{s}^{4}},\dfrac{1}{\sigma_{r}^{4}}\right\}\left(\|\mc{R}_{1}^{H}\n\mc{E}\n\mc{V}_{1}\|_{F}^{2}+\|\mc{T}_{1}^{H}\n\mc{E}^{H}\n\mc{U}_{2}\|_{F}^{2}+\|\mc{T}_{2}^{H}\n\mc{E}^{H}\n\mc{U}_{1}\|_{F}^{2}\right).\nonumber\\\label{eqnfrst}
\end{eqnarray}
Similarly, 
\begin{eqnarray}
    \|\mc{B}^{\dg}-\mc{A}^{\dg}\|_{F}^{2}&=&\left\|\begin{bmatrix}\mc{V}_{1}^{H}\\\mc{V}_{2}^{H}\end{bmatrix}\n\left(\mc{T}_{1}\n\mc{S}_{1}^{-1}\n\mc{R}_{1}^{H}-\mc{V}_{1}\n\mc{D}_{1}^{-1}\n\mc{U}_{1}^{H}\right)\n\begin{bmatrix}\mc{R}_{1}&\mc{R}_{2}\end{bmatrix}\right\|_{F}^{2}\nonumber\\
    &=&\left\|\begin{bmatrix}\mc{V}_{1}^{H}\n\mc{T}_{1}\n\mc{S}_{1}^{-1}-\mc{D}_{1}^{-1}\n\mc{U}_{1}^{H}\n\mc{R}_{1}&-\mc{D}_{1}^{-1}\n\mc{U}_{1}^{H}\n\mc{R}_{2}\\
    \mc{V}_{2}^{H}\n\mc{T}_{1}\n\mc{S}_{1}^{-1}&\mc{O}\end{bmatrix}\right\|_{F}^{2}\nonumber\\
    &=&\left\|\begin{bmatrix}\mc{D}_{1}^{-1}\n\mc{U}_{1}^{H}\n\mc{E}\n\mc{T}_{1}\n\mc{S}_{1}^{-1}&\mc{D}_{1}^{-2}\n\mc{V}_{1}^{H}\n\mc{E}^{H}\n\mc{R}_{2}\\
    \mc{V}_{2}^{H}\n\mc{E}^{H}\n\mc{R}_{1}\n\mc{S}_{1}^{-2}&\mc{O}\end{bmatrix}\right\|_{F}^{2}\nonumber\\
    &\leq&\max\left\{\dfrac{1}{\sigma_{r}^{2}\delta_{s}^{2}},\dfrac{1}{\delta_{s}^{4}},\dfrac{1}{\sigma_{r}^{4}}\right\}\left(\|\mc{U}_{1}^{H}\n\mc{E}\n\mc{T}_{1}\|_{F}^{2}+\|\mc{V}_{1}^{H}\n\mc{E}^{H}\n\mc{R}_{2}\|_{F}^{2}+\|\mc{V}_{2}^{H}\n\mc{E}^{H}\n\mc{R}_{1}\|_{F}^{2}\right).\nonumber\\\label{eqsecnd}
\end{eqnarray}
Thus, by adding inequalities \eqref{eqnfrst} and \eqref{eqsecnd}, we get
\begin{multline*}
    2\|\mc{B}^{\dg}-\mc{A}^{\dg}\|_{F}^{2}\leq\max\left\{\dfrac{1}{\sigma_{1}^{2}\delta_{s}^{2}},\dfrac{1}{\delta_{s}^{4}},\dfrac{1}{\sigma_{r}^{4}}\right\}\left(\|\mc{R}_{1}^{H}\n\mc{E}\n\mc{V}_{1}\|_{F}^{2}+\|\mc{T}_{1}^{H}\n\mc{E}^{H}\n\mc{U}_{2}\|_{F}^{2}+\|\mc{T}_{2}^{H}\n\mc{E}^{H}\n\mc{U}_{1}\|_{F}^{2}\right.\\\left.+\|\mc{U}_{1}^{H}\n\mc{E}\n\mc{T}_{1}\|_{F}^{2}+\|\mc{V}_{1}^{H}\n\mc{E}^{H}\n\mc{R}_{2}\|_{F}^{2}+\|\mc{V}_{2}^{H}\n\mc{E}^{H}\n\mc{R}_{1}\|_{F}^{2}\right)\\
    \leq 2\max\left\{\dfrac{1}{\sigma_{r}^{2}\delta_{s}^{2}},\dfrac{1}{\delta_{s}^{4}},\dfrac{1}{\sigma_{r}^{4}}\right\}\|\mc{E}\|_{F}^{2},
\end{multline*}
due to Lemma \ref{rpsntfnorm}. Therefore,
\begin{eqnarray*}
    \|\mc{B}^{\dg}-\mc{A}^{\dg}\|_{F}&\leq&\max\left\{\dfrac{1}{\sigma_{r}\delta_{s}},\dfrac{1}{\sigma_{r}^{2}},\dfrac{1}{\delta_{s}^{2}}\right\}\|\mc{E}\|_{F}\\
    &=&\max\left\{\|\mc{A}^{\dg}\|_{2}\|\mc{B}^{\dg}\|_{2},~\|\mc{A}^{\dg}\|_{2}^{2},~\|\mc{B}^{\dg}\|_{2}^{2}\right\}\|\mc{E}\|_{F}.
\end{eqnarray*}
\end{proof}
\end{theorem}

More on perturbation theory will appear in our next work.
We would further like to bring the attention of an interested reader to the very recent work by Liu and Zhen \cite{liu2019}, where the authors provided an algorithm to compute the spectral radius of a nonnegative tensor. This will help to find an iterative solution of a tensor multilinear system of the form $\mc{A}*_N \mc{X} = \mc{B}$ where $\mc{A} \in \mathbb{R}^{I_{1}\times\cdots\times I_{M} \times J_{1} \times\cdots\times J_{N}}$ and $\mc{B} \in \mathbb{R}^{I_{1}\times\cdots\times I_{M}}$ are given. To do this, we introduce below the definition of a sub-proper splitting of a tensor $\mc{A}$.

\begin{definition}
A splitting $\mc{A}=\mc{U}-\mc{V}$ of a tensor $\mc{A}\in{\C}^{I_{1}\times \cdots\times I_{M}\times J_{1}\times \cdots \times J_{N}}$ is called a sub-proper splitting if $\mc{U}\n\mc{U}^{\dg}\m\mc{A}=\mc{A}$ and $\mc{A}\n\mc{U}^{\dg}\m\mc{U}=\mc{A}$.
\end{definition}
The next result helps us to frame the iteration scheme for solving the tensor multilinear system  $\mc{A}\n \mc{X} = \mc{B}$.

\begin{theorem}
Let $\mc{A}=\mc{U}-\mc{V}$ be a sub-proper splitting of a tensor $\mc{A}\in{\C}^{I_{1}\times \cdots\times I_{M}\times J_{1}\times \cdots \times J_{N}}$. Then 
$\mc{A}=\mc{U}\n(\mc{I}-\mc{U}^{\dg}\m\mc{V})$.
    Further, if $1\notin \Omega(\mc{U}^{\dg}\m\mc{V})$, then
     \begin{enumerate}
    \item $\mc{A}^{\dg}=(\mc{I}-\mc{U}^{\dg}\m\mc{V})^{-1}\n\mc{U}^{\dg}$;
    \item If the system $\mc{A}\n \mc{X} = \mc{B}$ is consistent, then every solution $\mc{X}$ satisfying  $\mc{U}^{\dg}\m\mc{U}\n\mc{X}=\mc{X}$ (of which at least one exists, e.g., $\mc{A}^{\dg}\m\mc{B}$) satisfies $\mc{X}=\mc{U}^{\dg}\m\mc{V}\n\mc{X}+\mc{U}^{\dg}\m\mc{B}$. 
    \end{enumerate}
\end{theorem}
\begin{proof}
For a sub-proper splitting $\mc{A}=\mc{U}-\mc{V}$ of a tensor $\mc{A}\in{\C}^{I_{1}\times \cdots\times I_{M}\times J_{1}\times \cdots \times J_{N}}$, we have $\mc{U}\n\mc{U}^{\dg}\m\mc{V}=\mc{V}$ and  $\mc{V}\n\mc{U}^{\dg}\m\mc{U}=\mc{V}$. Thus, we obtain $\mc{A}=\mc{U}\n(\mc{I}-\mc{U}^{\dg}\m\mc{V})$.

The first case of the second part follows immediately from Theorem \ref{nthm22}. Also, let $\mc{X}_{0}$ be a solution of the consistent multilinear system $\mc{A}\n \mc{X} = \mc{B}$ which satisfies $\mc{U}^{\dg}\m\mc{U}\n\mc{X}_{0}= \mc{X}_{0}$. Thus, $\mc{U}\n\mc{X}_{0}=\mc{V}\n\mc{X}_{0}+\mc{B}$. Now, by pre-multiplying $\mc{U}^{\dg}$, we obtain $\mc{X}_{0}=\mc{U}^{\dg}\m\mc{V}\n\mc{X}_{0}+\mc{U}^{\dg}\m\mc{B}$.
\end{proof}
Therefore, the iteration scheme
\begin{equation}\label{spitrsch}
        \mc{X}_{k+1}=\mc{U}^{\dg}\m\mc{V}\n\mc{X}_{k}+\mc{U}^{\dg}\m\mc{B}.
    \end{equation}
will converge to $\mc{A}^{\dg}\m\mc{B}$, the least-squares solution of minimum norm of the tensor multilinear system $\mc{A}\n\mc{X}=\mc{B}$ if the spectral radius of $\mc{U}^{\dg}\m\mc{V}$ is less than $1$.

\section{Concluding remarks}
In this paper, we have  studied the problem of the reverse-order law for arbitrary order tensors. The advantage of this study is that the proofs do not use the notion of the range and the null space of a tensor. 
The important findings are summarized as
follows:

\begin{itemize}

\item The problem of the reverse-order law for arbitrary order tensors is taken up in Subsection \ref{secrolat}.

\vspace{.1cm}

\item Some new properties of a $\{1,2,3\}$-inverse and a $\{1,2,4\}$-inverse of a tensor are examined next in Subsection \ref{secrol1234}.

\vspace{.1cm}

  \item As an application of the reverse-order law, the problem of computing the additive Moore--Penrose inverse of tensors is partially attempted in subsection 3.2. The general case is left for future study and is posed as an open problem below.

 \item[] {\bf Problem:}  How to compute the Moore--Penrose inverse of the sum of two arbitrary order tensors $\mc{A}$ and $\mc{B}$?
\vspace{.1cm}

\item The notion of the Frobenius norm and the spectral norm for arbitrary order tensors are then recalled to analyze the additive perturbation bounds for the Moore–Penrose inverse of arbitrary order tensors via the Einstein product by means of Frobenius norm.  More on additive perturbation bounds using the spectral norm and the $Q$-norm will appear in our next article.

    \item Finally, we have introduced the notion of the sub-proper splitting of a tensor to find an iterative solution of a tensor multilinear system. The detailed convergence analysis of the iteration scheme  \eqref{spitrsch} can be further analyzed by the help of the reverse-order law, and is left for a future topic of research interest.
\end{itemize}
\newpage
\vspace{.6cm}
\noindent
{\small {\bf Acknowledgments.}\\
We thank the {\bf Government of India} for introducing the {\it work from home initiative} during the COVID-19 pandemic.
This research was supported by grant CRG/2018/002986 from the  Science and Engineering Research Board, Department of Science and Technology, New Delhi, India. }

\bibliographystyle{amsplain}

\end{document}